\documentclass[reqno]{amsart}

\usepackage[latin1]{inputenc}
\usepackage{amssymb}
\usepackage{graphicx}
\usepackage{amscd}
\usepackage[hidelinks]{hyperref}
\usepackage{color}
\usepackage{float}
\usepackage{graphics,amsmath,amssymb}
\usepackage{amsthm}
\usepackage{amsfonts}
\usepackage{latexsym}
\usepackage{epsf}
\usepackage{enumerate}
\usepackage{xifthen}
\usepackage{mathrsfs}
\usepackage{dsfont}
\usepackage{makecell}
\usepackage[FIGTOPCAP]{subfigure}
\usepackage{amsmath}
\allowdisplaybreaks[4]
\usepackage{listings}
\usepackage{etoolbox}
\usepackage{fancyhdr}
\usepackage{pdflscape}
\usepackage[title,toc,titletoc]{appendix}
\usepackage{enumitem}
\usepackage{cite}

 \newtheoremstyle{mytheorem}
 {3pt}
 {3pt}
 {\slshape}
 {}
 {\bfseries}
 {.}
 { }
 {}

\numberwithin{equation}{section}

\theoremstyle{theorem}
\newtheorem{theorem}{Theorem}[section]

\newtheorem{lemma}[theorem]{Lemma}

\theoremstyle{definition}

\newtheorem{remark}{Remark}[section]

\newcommand{\Keywords}[1]{\ifthenelse{\isempty{#1}}{}{\smallskip \smallskip \noindent \textbf{Keywords}. #1}}
\newcommand{\MSC}[2][2010]{\ifthenelse{\isempty{#2}}{}{\smallskip \smallskip \noindent \textbf{#1MSC}. #2}}
\newcommand{\abstractnote}[1]{\ifthenelse{\isempty{#1}}{}{\smallskip \smallskip \noindent \textsuperscript{\dag}#1}}

\makeatletter
\def\specialsection{\@startsection{section}{1}%
  \z@{\linespacing\@plus\linespacing}{.5\linespacing}%
  {\normalfont}}
\def\section{\@startsection{section}{1}%
  \z@{.7\linespacing\@plus\linespacing}{.5\linespacing}%
  {\normalfont\scshape}}
\patchcmd{\@settitle}{\uppercasenonmath\@title}{\Large\boldmath}{}{}
\patchcmd{\@settitle}{\begin{center}}{\begin{flushleft}}{}{}
\patchcmd{\@settitle}{\end{center}}{\end{flushleft}}{}{}
\patchcmd{\@setauthors}{\MakeUppercase}{\normalsize}{}{}
\patchcmd{\@setauthors}{\centering}{\raggedright}{}{}
\patchcmd{\section}{\scshape}{\large\bfseries\boldmath}{}{}
\patchcmd{\subsection}{\bfseries}{\bfseries\boldmath}{}{}
\renewcommand{\@secnumfont}{\bfseries}
\patchcmd{\@startsection}{\@afterindenttrue}{\@afterindentfalse}{}{}
\patchcmd{\abstract}{\leftmargin3pc}{\leftmargin1pc}{}{}

\def\maketitle{\par
  \@topnum\z@ 
  \@setcopyright
  \thispagestyle{empty}
  \ifx\@empty\shortauthors \let\shortauthors\shorttitle
  \else \andify\shortauthors
  \fi
  \@maketitle@hook
  \begingroup
  \@maketitle
  \toks@\@xp{\shortauthors}\@temptokena\@xp{\shorttitle}%
  \toks4{\def\\{ \ignorespaces}}
  \edef\@tempa{%
    \@nx\markboth{\the\toks4
      \@nx\MakeUppercase{\the\toks@}}{\the\@temptokena}}%
  \@tempa
  \endgroup
  \c@footnote\z@
  \@cleartopmattertags
}
\makeatother

\newcommand{\cm}{\mathscr{M}}
\newcommand{\op}{\overline{p}}
\newcommand{\pd}{\mathrm{pod}}
\newcommand{\oM}{\overline{M}}

\newcommand{\qbinom}[2]{\begin{bmatrix}#1\\#2\end{bmatrix}}

\newcommand{\Log}{\mathrm{Log}}

\title[Truncated pentagonal number theorem]{A further look at the truncated pentagonal number theorem}

\author[S. Chern]{Shane Chern}
\address{Department of Mathematics, The Pennsylvania State University, University Park, PA 16802, USA}
\email{shanechern@psu.edu}

\date{}

\begin{document}

%

\maketitle

\begin{abstract}

In this paper, we study the asymptotic behavior of the following function
$$M_k(n):=(-1)^{k-1} \sum_{j=0}^{k-1}\big(p(n-j(3j+1)/2)-p(n-j(3j+5)/2-1)\big),$$
which arises from Andrews and Merca's truncated pentagonal number theorem.


\MSC{11P82, 05A17.}
\end{abstract}

\section{Introduction}

In \cite{AM2012}, Andrews and Merca studied a truncated version of Euler's pentagonal number theorem. The motivation of their work arises from the non-negativity of the following function:
\begin{equation}
M_k(n):=(-1)^{k-1} \sum_{j=0}^{k-1}(-1)^j\big(p(n-j(3j+1)/2)-p(n-j(3j+5)/2-1)\big),
\end{equation}
where $n$ and $k$ are positive integers, and $p(n)$ denotes the number of partitions of $n$ \cite{And1976}. Andrews and Merca also gave a pratition-theoretic interpretation of $M_k(n)$. Namely, it denotes the number of partitions of $n$ in which $k$ is the least integer that is not a part and there are more parts $>k$ than there are $<k$.

Their proof of the non-negativity of $M_k(n)$ relies on a clever reformulation of the generating function of $M_k(n)$. Namely, if we put $M_k(0)=(-1)^{k-1}$, then
\begin{align}
\cm_k(q):=\sum_{n\ge 0}M_k(n)q^n=&\frac{(-1)^{k-1}}{(q;q)_\infty} \sum_{j=0}^{k-1} (-1)^j q^{j(3 j+1)/2} (1-q^{2j+1})\label{eq:gf1}\\
=&(-1)^{k-1}+\sum_{n\ge 1} \frac{q^{\binom{n}{2}+(k+1)n}}{(q;q)_n} \qbinom{n-1}{k-1}_q,\label{eq:gf2}
\end{align}
where
$$(A;q)_n = \prod_{j=0}^{n-1} (1-Aq^j),$$
and
$$\qbinom{A}{B}_q=\begin{cases}
0, & \text{if $B<0$ or $B>A$},\\
\frac{(q;q)_A}{(q;q)_B (q;q)_{A-B}}, & \text{otherwise}.
\end{cases}$$
One immediately sees that the non-negativity of $M_k(n)$ for $n\ge 1$ follows from \eqref{eq:gf2}.

Now, if one fixes $k$ and computes some values of $M_k(n)$, one may notice that $M_k(n)$ grows rapidly as $n$ becomes large. This stimulates us to study the asymptotic behavior of $M_k(n)$. In this paper, we shall show

\begin{theorem}
Let $\epsilon>0$ be arbitrary small. Then as $n\to\infty$, we have, for $k\ll n^{\frac{1}{8}-\epsilon}$,
\begin{equation}\label{eq:main}
M_k(n)=\frac{\pi}{12\sqrt{2}}kn^{-\frac{3}{2}}e^{\frac{2\pi\sqrt{n}}{\sqrt{6}}}+O\left(k^3n^{-\frac{7}{4}}e^{\frac{2\pi\sqrt{n}}{\sqrt{6}}}\right).
\end{equation}
\end{theorem}

\begin{remark}
Here the assumption $k\ll n^{\frac{1}{8}-\epsilon}$ ensures that $O\left(k^3n^{-\frac{7}{4}}e^{\frac{2\pi\sqrt{n}}{\sqrt{6}}}\right)$ is indeed an error term.
\end{remark}

Apparently, \eqref{eq:main} demonstrates the positivity of $M_k(n)$ for sufficiently large $n$ if we fix $k$. In fact, this asymptotic formula allows us to have a better understanding of $M_k(n)$. The interested reader may also compare \eqref{eq:main} with the celebrated asymptotic expression for $p(n)$ due to Hardy and Ramanujan \cite{HR1918}
$$p(n)\sim \frac{1}{4\sqrt{3}}n^{-1}e^{\frac{2\pi\sqrt{n}}{\sqrt{6}}}.$$

\section{Proof}\label{sect:proof}

Throughout this section, we let $q=e^{2\pi i\tau}$ with $\tau=x+iy\in\mathbb{H}$ (i.e. $y>0$). We also put
$$y=\frac{1}{2\sqrt{6n}}\qquad\text{and}\qquad M=\sqrt{\left(\frac{12}{12-\pi^2}\right)^2-1}.$$
Note that we may take $M$ to be other (positive) absolute constant. However, we choose the above value for computational convenience.

\subsection{Asymptotics of $\cm_k(q)$ near $q=1$}

We first estimate $\cm_k(q)$ near $q=1$.

\begin{lemma}\label{le:near1}
For $|x|\le My$, we have, as $n\to\infty$ (and hence $y\to 0^+$),
\begin{equation}\label{eq:near1}
\cm_k(q)=-2e^{\frac{\pi i}{4}}\pi k\tau^{\frac{3}{2}}e^{\frac{\pi i}{12\tau}}+O\left(k^3 n^{-\frac{5}{4}} e^{\frac{\pi \sqrt{n}}{\sqrt{6}}}\right).
\end{equation}
\end{lemma}

\begin{proof}
We have
$$1-q^{2j+1}=1-e^{2(2j+1)\pi i\tau}=-2(2j+1)\pi i\tau +\mathcal{E}_j,$$
where
\begin{align*}
|\mathcal{E}_j|&=\left|e^{2(2j+1)\pi i\tau}-1-2(2j+1)\pi i\tau\right|\\
&\le e^{|2(2j+1)\pi i\tau|}-1-|2(2j+1)\pi i\tau|\\
&\le 4(2j+1)^2\pi^2 |\tau|^2,
\end{align*}
since $|2(2j+1)\pi i\tau|<1$ for $0\le j\le k-1$ (which is ensured by the assumption $k\ll n^{\frac{1}{8}-\epsilon}$) whereas $e^x-1-x\le x^2$ when $0<x<1$. Hence,
\begin{align*}
\sum_{j=0}^{k-1} (-1)^j q^{j(3 j+1)/2} (1-q^{2j+1})&=\sum_{j=0}^{k-1}(-1)^j\big(-2(2j+1)\pi i\tau\big)+\mathcal{E}\\
&=(-1)^k 2\pi i k\tau+\mathcal{E},
\end{align*}
where
$$|\mathcal{E}|\le\sum_{j=0}^{k-1} 4(2j+1)^2\pi^2|\tau|^2\ll k^3y^2.$$
Consequently, we have
\begin{equation}\label{eq:near1-part1}
\sum_{j=0}^{k-1} (-1)^j q^{j(3 j+1)/2} (1-q^{2j+1})=(-1)^k 2\pi i k\tau+O(k^3 y^2).
\end{equation}

Furthermore, we know from the modular inversion formula for Dedekind's eta-function (cf.~\cite[p.~121,  Proposition 14]{Kob1984}) that
\begin{equation}\label{eq:eta-trans}
(q;q)_\infty = \frac{1}{\sqrt{-i\tau}} e^{-\frac{\pi i\tau}{12}-\frac{\pi i}{12\tau}}\bigg(1+O\left(e^{-\frac{2\pi i}{\tau}}\right)\bigg),
\end{equation}
where the square root is taken on the principal branch, with $z^{1/2}>0$ for $z>0$. Hence,
\begin{equation}\label{eq:near1-part2}
\frac{1}{(q;q)_\infty}=\sqrt{-i\tau}e^{\frac{\pi i}{12\tau}}+O\left(y^{\frac{3}{2}}e^{\frac{\pi}{12}\Im\left(\frac{-1}{\tau}\right)}\right).
\end{equation}

Finally, \eqref{eq:near1} follows from \eqref{eq:gf1}, \eqref{eq:near1-part1}, \eqref{eq:near1-part2} and the fact that
$$\Im\left(\frac{-1}{\tau}\right)=\frac{y}{x^2+y^2}\le \frac{1}{y}.$$
This finishes the proof of Lemma \ref{le:near1}.
\end{proof}

\subsection{Asymptotics of $\cm_k(q)$ away from $q=1$}

We next estimate $\cm_k(q)$ away from $q=1$.

\begin{lemma}\label{le:away1}
For $My<|x|\le \frac{1}{2}$, we have, as $n\to\infty$ (and hence $y\to 0^+$),
\begin{equation}\label{eq:away1}
\cm_k(q)\ll k n^{-\frac{1}{4}}e^{\frac{\pi \sqrt{n}}{2\sqrt{6}}}.
\end{equation}
\end{lemma}

\begin{proof}
We first have the following trivial bound
\begin{equation}\label{eq:away1-part1}
\left|\sum_{j=0}^{k-1} (-1)^j q^{j(3 j+1)/2} (1-q^{2j+1})\right|\le 2k.
\end{equation}

On the other hand,
\begin{align*}
\Log\left(\frac{1}{(q;q)_\infty}\right)&=-\sum_{n\ge 1}\Log(1-q^n)\\
&=\sum_{n\ge 1}\sum_{m\ge 1}\frac{q^{nm}}{m}\\
&=\sum_{m\ge 1}\frac{q^m}{m(1-q^m)}.
\end{align*}
Hence,
\begin{align}
\left|\Log\left(\frac{1}{(q;q)_\infty}\right)\right|&\le \sum_{m\ge 1}\frac{|q|^m}{m|1-q^m|}\nonumber\\
&\le \sum_{m\ge 1}\frac{|q|^m}{m(1-|q|^m)}-\frac{|q|}{1-|q|}+\frac{|q|}{|1-q|}\nonumber\\
&=\Log\left(\frac{1}{(|q|;|q|)_\infty}\right)-|q|\left(\frac{1}{1-|q|}-\frac{1}{|1-q|}\right).\label{eq:away1-part21}
\end{align}
It follows from \eqref{eq:eta-trans} that
\begin{equation}\label{eq:away1-part22}
\frac{1}{(|q|;|q|)_\infty}=\sqrt{y} e^{\frac{\pi}{12y}}\Bigg(1+O\left(e^{-\frac{2\pi}{y}}\right)\Bigg).
\end{equation}
Furthermore, we know from the fact $|x|>My$ that $\cos(2\pi x)<\cos(2\pi My)\le 1$. Hence,
\begin{align*}
|1-q|^2&= 1-2e^{-2\pi y} \cos(2\pi x)+e^{-4\pi y}\\
&>1-2e^{-2\pi y} \cos(2\pi My)+e^{-4\pi y}.
\end{align*}
Computing the Taylor expansion around $y=0$ yields
\begin{equation}\label{eq:away1-part23}
|1-q|>2\pi y\sqrt{1+M^2}+O(y^2).
\end{equation}
Using the fact $1-|q|=2\pi y+O(y^2)$ and combining \eqref{eq:away1-part21}, \eqref{eq:away1-part22} and \eqref{eq:away1-part23}, we conclude that
\begin{equation}\label{eq:away1-part2}
\left|\frac{1}{(q;q)_\infty}\right|\ll \sqrt{y} \exp\Bigg(\frac{1}{y}\left(\frac{\pi}{12}-\frac{1}{2\pi} \left(1-\frac{1}{\sqrt{1+M^2}}\right)\right)\Bigg).
\end{equation}

Finally, \eqref{eq:away1} follows from \eqref{eq:away1-part1} and \eqref{eq:away1-part2}.
\end{proof}

\subsection{Applying Wright's circle method}

Let $\mathcal{C}$ denote the circle $q=e^{2\pi i\tau}=e^{2\pi i (x+iy)}$ where $x\in[-\frac{1}{2},\frac{1}{2}]$. Cauchy's integral formula tells us that
\begin{align}
M_k(n)&=\frac{1}{2\pi i}\oint_{\mathcal{C}} \frac{\cm_k(q)}{q^{n+1}} \;dq\nonumber\\
&=\int_{-\frac{1}{2}}^{\frac{1}{2}}\cm_k\left(e^{2\pi i\tau}\right)e^{-2n\pi i\tau}\;dx\nonumber\\
&=\int_{|x|\le My}+\int_{My<|x|\le \frac{1}{2}} =:\mathcal{I}_1+\mathcal{I}_2,
\end{align}
where the integrands in $\mathcal{I}_1$ and $\mathcal{I}_2$ are both $\cm_k\left(e^{2\pi i\tau}\right)e^{-2n\pi i\tau}$.

We first compute $\mathcal{I}_1$, which contributes to the main term. Our evaluation relies on a function $P_s(u)$ defined by Wright \cite{Wri1934}. For fixed $M>0$ and $u\in\mathbb{R}_{>0}$, let
$$P_s(u):=\frac{1}{2\pi i}\int_{1-Mi}^{1+Mi} v^s e^{u\left(v+\frac{1}{v}\right)}\;dv.$$
Wright \cite[p.~138, Lemma XVII]{Wri1934} showed that this function can be rewritten in terms of the $I$-Bessel function up to an error term.

\begin{lemma}[Wright]
We have, as $u\to\infty$,
\begin{equation}
P_s(u)=I_{-s-1}(2u)+O(e^u),
\end{equation}
where $I_\ell$ denotes the usual $I$-Bessel function of order $\ell$.
\end{lemma}

\noindent We also recall that the asymptotic expansion of $I_\ell(x)$ (cf.~\cite[p.~377, (9.7.1)]{AA1972}) states that, for fixed $\ell$, when $|\arg x|<\frac{\pi}{2}$,
\begin{align}\label{Bessel-order}
I_{\ell}(x)\sim \frac{e^x}{\sqrt{2\pi x}}\left(1-\frac{4\ell^2-1}{8x}+\frac{(4\ell^2-1)(4\ell^2-9)}{2!(8x)^2}-\cdots \right).
\end{align}

It follows from Lemma \ref{le:near1} that
$$\mathcal{I}_1=\int_{|x|\le My} e^{-2n\pi i\tau}\Bigg(-2e^{\frac{\pi i}{4}}\pi k\tau^{\frac{3}{2}}e^{\frac{\pi i}{12\tau}}+O\left(k^3 n^{-\frac{5}{4}} e^{\frac{\pi \sqrt{n}}{\sqrt{6}}}\right)\Bigg)\;dx.$$
Making the change of variables $v=-i\tau/y$ yields
\begin{align}
\mathcal{I}_1&=\int_{1-Mi}^{1+Mi}(-iy)e^{2n\pi yv}\Bigg(-2e^{\frac{\pi i}{4}}\pi k (iyv)^{\frac{3}{2}} e^{\frac{\pi}{12yv}}+O\left(k^3 n^{-\frac{5}{4}} e^{\frac{\pi \sqrt{n}}{\sqrt{6}}}\right)\Bigg)\; dv\nonumber\\
&=2^{-\frac{7}{4}} 3^{-\frac{5}{4}} \pi^2 k n^{-\frac{5}{4}}P_{\frac{3}{2}}\left(\frac{\pi \sqrt{n}}{\sqrt{6}}\right)+O\left(k^3 n^{-\frac{7}{4}} e^{\frac{2\pi \sqrt{n}}{\sqrt{6}}}\right)\nonumber\\
&=2^{-\frac{7}{4}} 3^{-\frac{5}{4}} \pi^2 k n^{-\frac{5}{4}} I_{-\frac{5}{2}}\left(\frac{2\pi \sqrt{n}}{\sqrt{6}}\right)+O\left(k^3 n^{-\frac{7}{4}} e^{\frac{2\pi \sqrt{n}}{\sqrt{6}}}\right)\nonumber\\
&=\frac{\pi}{12\sqrt{2}}kn^{-\frac{3}{2}}e^{\frac{2\pi\sqrt{n}}{\sqrt{6}}}+O\left(k^3n^{-\frac{7}{4}}e^{\frac{2\pi\sqrt{n}}{\sqrt{6}}}\right).\label{eq:I1}
\end{align}

We now evaluate $\mathcal{I}_2$. It follows from Lemma \ref{le:away1} that
\begin{align}
\mathcal{I}_2\ll \int_{My<|x|\le \frac{1}{2}} k n^{-\frac{1}{4}}e^{\frac{\pi \sqrt{n}}{2\sqrt{6}}} e^{\frac{\pi \sqrt{n}}{\sqrt{6}}}\;dx \ll k n^{-\frac{1}{4}}e^{\frac{3\pi \sqrt{n}}{2\sqrt{6}}}.\label{eq:I2}
\end{align}

Consequently, we know from \eqref{eq:I1} and \eqref{eq:I2} that as $n\to\infty$,
$$M_k(n)=\frac{\pi}{12\sqrt{2}}kn^{-\frac{3}{2}}e^{\frac{2\pi\sqrt{n}}{\sqrt{6}}}+O\left(k^3n^{-\frac{7}{4}}e^{\frac{2\pi\sqrt{n}}{\sqrt{6}}}\right).$$
This is our main result.

\section{Closing remarks}

There are more truncated theta series identities. Two interesting examples are due to Guo and Zeng \cite{GZ2013}:
\begin{equation}
\frac{(-q;q)_\infty}{(q;q)_\infty}\sum_{j=-k}^k (-1)^j q^{j^2} = 1+ (-1)^k \sum_{n\ge k+1} \frac{(-q;q)_k (-1;q)_{n-k} q^{(k+1)n}}{(q;q)_n}\qbinom{n-1}{k}_q,
\end{equation}
and
\begin{align}
&\frac{(-q;q^2)_\infty}{(q^2;q^2)_\infty}\sum_{j=0}^{k-1}(-1)^j q^{j(2j+1)}(1-q^{2j+1})\nonumber\\
&\qquad=1+(-1)^{k-1}\sum_{n\ge k}\frac{(-q;q^2)_k (-q;q^2)_{n-k} q^{2(k+1)n-k}}{(q^2;q^2)_n}\qbinom{n-1}{k-1}_{q^2}.
\end{align}

Let $\op(n)$ denote the number of overpartitions of $n$ (i.e.~partitions of $n$ where the first occurrence of each distinct part may be overlined) and let $\pd(n)$ denote the number of partitions of $n$ wherein odd parts are not repeated. The above two identities respectively reveal the non-negativity of the following two functions (the notation of which is due to Andrews and Merca \cite{AM2018}) for $n,k\ge 1$:
\begin{equation}
\oM_k(n):=(-1)^k \sum_{j=-k}^k (-1)^j \op(n-j^2),
\end{equation}
and
\begin{equation}
MP_k(n):=(-1)^{k-1} \sum_{j=0}^{k-1} (-1)^j\big(\pd(n-j(2j+1))-\pd(n-(j+1)(2j+1))\big).
\end{equation}
In \cite{AM2018}, to answer a question of Guo and Zeng \cite[p.~702]{GZ2013}, Andrews and Merca also presented the partition-theoretic interpretations of $\oM_k(n)$ and $MP_k(n)$:
\begin{itemize}[leftmargin=0.15in]
\item $\oM_k(n)$ denotes the number of overpartitions of $n$ in which the first part larger than $k$
appears at least $k + 1$ times;
\item $MP_k(n)$ denotes the number of partitions of $n$ in which the first part larger than $2k - 1$ is odd and appears exactly $k$ times whereas all other odd parts appear at most once.
\end{itemize}

Using similar arguments to that in Sect.~\ref{sect:proof}, we are also able to show the asymptotic behaviors of $\oM_k(n)$ and $MP_k(n)$.

\begin{theorem}
Let $\epsilon>0$ be arbitrary small. Then as $n\to\infty$, we have, for $k\ll n^{\frac{1}{12}-\epsilon}$,
\begin{equation}
\oM_k(n)=\frac{1}{8}n^{-1}e^{\pi\sqrt{n}}+O\left(k^3n^{-\frac{5}{4}}e^{\pi \sqrt{n}}\right).
\end{equation}
\end{theorem}

\begin{remark}
It is interesting to point out that the main term of $\oM_k(n)$ is identical to the main term in the asymptotic expression of $\op(n)$.
\end{remark}

\begin{theorem}
Let $\epsilon>0$ be arbitrary small. Then as $n\to\infty$, we have, for $k\ll n^{\frac{1}{8}-\epsilon}$,
\begin{equation}
MP_k(n)=\frac{\pi}{16}kn^{-\frac{3}{2}}e^{\frac{\pi\sqrt{n}}{\sqrt{2}}}+O\left(k^3n^{-\frac{7}{4}}e^{\frac{\pi\sqrt{n}}{\sqrt{2}}}\right).
\end{equation}
\end{theorem}

\subsection*{Acknowledgements}

I would like to thank George Andrews for some helpful suggestions.

\bibliographystyle{amsplain}

\end{document}